\documentclass{amsart}
\usepackage{amssymb,amsmath,amsthm,amscd}

\newtheorem{theorem}{Theorem}[section]
\newtheorem{lemma}[theorem]{Lemma}
\newtheorem{proposition}[theorem]{Proposition}
\newtheorem{corollary}[theorem]{Corollary}

\theoremstyle{definition}
\newtheorem{definition}[theorem]{Definition}
\newtheorem{example}[theorem]{Example}

\newtheorem{question}[theorem]{Question}

\theoremstyle{remark}
\newtheorem{remark}[theorem]{Remark}

\numberwithin{equation}{section}

\begin{document}

\title{Strong Shift Equivalence and Positive  Doubly Stochastic Matrices}
\author{Sompong Chuysurichay}
\address{Department of Mathematics and Statistics, Prince of Songkla University, Songkhla, Thailand 90110}
\email{sompong.c@psu.ac.th}
\date{\today.}

\begin{abstract}
We give sufficient conditions for a positive stochastic matrix to be similar and strong shift equivalent over $\mathbb{R}_+$ to a positive doubly stochastic matrix through matrices of the same size. We also prove that every positive stochastic matrix is strong shift equivalent over $\mathbb{R}_+$ to a positive doubly stochastic matrix. Consequently, the set of nonzero spectra of primitive stochastic matrices over $\mathbb{R}$ with positive trace and the set of nonzero spectra of positive doubly stochastic matrices over $\mathbb{R}$ are identical. We exhibit a class of $2\times 2$ matrices, pairwise strong shift equivalent over $\mathbb R_+$ through $2\times 2$ matrices, for which there is no uniform upper bound on the minimum lag of a strong shift equivalence through matrices of bounded size. In contrast, we show for any $n\times n$ primitive matrix of positive trace that the set of positive $n\times n$ matrices similar to it contains only finitely  many SSE-$\mathbb R_+$ classes.
\end{abstract}
\maketitle


\section{Introduction}
Strong shift equivalence theory (for matrices over $\mathbb Z_+$)
was introduced by R.F. Williams  \cite{williams1973classification}
as a tool for classifying shifts of finite type. Subsequently, strong shift equivalence
over other semirings has been used for classification of other symbolic dymamical systems
such as SFTs with Markov measure \cite{parry1982stochastic,marcus1991weight}
and SFTs with a free finite group action \cite{boyle2005equivariant}.

Despite its good-looking definition, strong shift equivalence is still very difficult to fully understand. Williams also introduced a more
tractable equivalence relation called shift equivalence and conjectured that shift equivalence and strong shift equivalence over
$\mathbb{Z}_+$ are the same. The conjecture was proved false by K.H. Kim and F.W. Roush in 1992 (reducible case)
\cite{kim1992williams} and 1997 (irreducible case) \cite{kim1999williams}. Although Williams' conjecture is false in
general, the gap between shift equivalence and strong shift equivalence over $\mathbb{Z}_+$ remains mysterious. We
study Williams' conjecture by relaxing the problem to the level of positive rational and real matrices, as in  \cite{S22,S24,S26,boyle2013path}.
We expect that the Williams' conjecture is true for positive rational (or real) matrices. This is the conjecture posed by Mike Boyle in \cite[Conjecture $5.1$]{boyle2008contemporary}. Understanding this relation is a natural step toward understanding strong shift equivalence over $\mathbb{Z}_+$,
and a natural matrix problem independently.

The purpose of this paper is twofold. Firstly, it gives a connection between stochastic matrices and doubly stochastic matrices via strong
shift equivalence, in both local and global aspects. In section $3$, we give sufficient conditions for an $n\times n$ positive stochastic matrix $P$ over a subfield $\mathbb{U}$ of $\mathbb{R}$ to be strong shift equivalent over $\mathbb{U}_+$ to a positive doubly stochastic matrix through matrices
of the same size. By conjugating $P$ with some involution, we show that if $P+J_n(I_n-P)$ is also positive ($J_n$ is a matrix
all of whose entries are $\frac{1}{n}$) then $P$ is similar and strong shift equivalent over $\mathbb{U}_+$ to a doubly stochastic
matrix through matrices of the same size.  It is not true, however, that every positive stochastic matrix
is strong shift equivalent over $\mathbb{R}_+$ to a doubly stochastic matrix of the same size: a counterexample was found by Johnson
\cite{johnson1981row}. In section $4$, we prove that any positive stochastic matrix is strong shift equivalent over $\mathbb{U}_+$ to a
positive doubly stochastic matrix, assuming $\mathbb{U}$ is a subring of $\mathbb{R}$ containing $\mathbb{Q}$.  Consequently, the set of
nonzero spectra of positive doubly stochastic matrices over a subfield $\mathbb{U}$ of $\mathbb{R}$ and the set of nonzero spectra of
primitive stochastic matrices over $\mathbb{U}$ with positive trace coincide. We do not know whether the result can be extended to
irreducible matrices.

The second aim of the paper is to give results and counterexamples for natural finiteness questions involving strong shift equivalence over
$\mathbb R_+$. We prove that strong shift equivalence over $\mathbb{R}_+$ of irreducible stochastic matrices can be studied using
generalized stochastic matrices in section $5$. We give  a family of $2\times 2$ matrices, pairwise strong shift equivalent over $\mathbb{R}_+$, such that there is no uniform bound on the minimum lag of a strong shift equivalence over $\mathbb{R}_+$ through matrices of bounded size between members of the family. In contrast, we prove that the collection of positive $n\times n$ matrices similar over $\mathbb{R}$ to any primitive $n\times n$ matrix of positive trace over $\mathbb{R}$ contains only finitely many SSE$-\mathbb{R}_+$ classes.
\section{Definitions and Background}
Let $A=(a_{ij})$ be a real $m\times n$ matrix. $A$ is \emph{nonnegative} if $a_{ij}\geq 0$ for all $i,j\in\{1,2,\ldots,n\}$, and $A$ is said to be \emph{positive} if $a_{ij}>0$ for all $i,j\in\{1,2,\ldots,n\}$. $A$ is \emph{irreducible} if $A$ is nonnegative, square, and for any $(i,j)$ there is some $k\in\mathbb{N}$ such that $(A^k)_{ij}>0$. $A$ is \emph{primitive} if $A$ is nonnegative, square, and there is some $r\in\mathbb{N}$ such that $A^r$ is positive. $A$ is \emph{stochastic} if it is nonnegative, square, and every row sum of $A$ is $1$. $A$ is \emph{doubly stochastic} if  $A$ and $A^T$ are both stochastic. We denote the identity matrix by $I_n$.

Let $\Delta^{n-1}_+$ denote the set of row vectors
$l=(l_1,l_2,\ldots,l_n)$ where $l$ is positive and $\sum_{k=1}^n l_k=1$. A
vector $l\in\Delta^{n-1}_+$   is called the \emph{left Perron
  eigenvector} of an $n\times n$ stochastic matrix $P$ if $lP=l$.  Let
$j_n$ denote the row vector in $\mathbb{R}^n$ whose all entries are $1$. For any $l\in \Delta^{n-1}_+$, let $J_l$ be an $n\times n$ matrix whose the $(i,j)$th entry is $l_j$. If $l=\frac{1}{n}j_n$, $J_l$ is also denoted by $J_n$.  If $\mathbb{U}$ is a subring of $\mathbb{R}$ we denote by $\mathbb{U}_+$ the subsemiring of $\mathbb{U}$ consisting of all nonnegative elements in $\mathbb{U}$.

\begin{definition}
  Let $A$ and $B$ be square matrices over a semiring $\mathcal{R}$ containing $0$ and $1$ as the additive and multiplicative identities.
  \begin{enumerate}
    \item $A$ is \emph{elementary strong shift equivalent over $\mathcal{R}$} (ESSE$-\mathcal{R}$) to $B$ if there exist matrices $U,V$ over $\mathcal{R}$ with $A=UV,B=VU$.
    \item $A$ is \emph{strong shift equivalent over $\mathcal{R}$} (SSE$-\mathcal{R}$) to $B$ if there exists a finite sequence of matrices over $\mathcal{R}$, $A=A_0,A_1,\ldots,A_l=B$, such that $A_i$ is ESSE$-\mathcal{R}$ to $A_{i+1}$ for all $i\in\{0,1,\ldots,l-1\}$. Such a finite sequence is a \emph{strong shift equivalence} over $\mathcal{R}$.  The number $l$ is the \emph{lag} of the strong shift equivalence. By the \emph{size} of the strong shift equivalence, we mean $\max\{n_i\mid 0\leq i\leq l, A_i\,\,\text{is}\,\,n_i\times n_i\}$.
    \item $A$ is \emph{shift equivalent over $\mathcal{R}$}
      (SE$-\mathcal{R}$) to $B$ if there exist matrices $U,V$ over
      $\mathcal{R}$ and $l\in\mathbb{N}$ such that $A^l=UV,B^l=VU$ and
      $AU=UB, VA=BV$. The number $l$ is the \emph{lag} of the shift equivalence.
  \end{enumerate}
\end{definition}

For any semiring $\mathcal{R}$, SSE$-\mathcal{R}$ and SE$-\mathcal{R}$
are equivalence relations whereas ESSE$-\mathcal{R}$ is not transitive. For all the semirings $\mathcal{R}$ under our
consideration the implication cannot be reversed.
If $(U_i,V_i)_{i=1}^{\ell}$ is a lag $\ell$  SSE$-\mathcal{R}$ from $A$ to $B$,
then $(U,V)=(U_1U_2\cdots U_{\ell},V_{\ell}V_{\ell-1}\cdots V_1)$ is a lag
$\ell$ SE$-\mathcal R$ from $A$ to $B$.

Next, we recall some basic constructions from \cite{boyle2013path}.

\begin{definition}
  An \emph{amalgamation matrix} is a matrix with entries from $\{0,1\}$ such that every row has exactly one $1$ and every column has at least one $1$. A \emph{subdivision matrix} is the transpose of an amalgamation matrix.
\end{definition}

\begin{definition}
  An \emph{elementary row splitting} is an elementary strong shift equivalence $A=UX,C=XU$ in which $U$ is a subdivision matrix. In this case, $C$ is an \emph{elementary row splitting of $A$}, and $A$ is an \emph{elementary row amalgamation of $C$}.
  \end{definition}

  \begin{definition}
  An \emph{elementary column splitting} is an elementary strong shift equivalence $A=XV,C=VX$ in which $V$ is an amalgamation matrix. In this case, $C$ is an \emph{elementary column splitting of $A$}, and $A$ is an \emph{elementary column amalgamation of $C$}.
  \end{definition}

A special case of an elementary column splitting is the following example.
  \begin{example}
    Let $A$ be an $n\times n$ nonnegative matrix. Let $X$ be an $n\times (n+1)$ matrix obtained by splitting column $i$ of $A$ into columns $i$ and $i+1$ and the other columns of $A$ and $X$ are the same. We duplicate row $i$ of $X$ and form an $(n+1)\times (n+1)$ matrix $C$. Let $V$ be an $(n+1)\times n$ matrix obtained by duplicating row $i$ of $I_n$. Then $A=XV$ and $C=VX$. Thus $C$ is an elementary column splitting of $A$.
  \end{example}
  We will only use elementary column splittings of this type in section $4$.

  \begin{definition}
Let $A$ be a square matrix. The \emph{Jordan form away from zero of $A$}, $J^{\times}(A)$, is the matrix obtained by removing from the Jordan form of $A$ all rows and columns with zeros on the main diagonal. The \emph{nonzero spectrum} of $A$ is the set of all nonzero eigenvalues of $A$.
\end{definition}

It is known that SE$-\mathbb{R}_+$ preserves Jordan forms away from zero and so does SSE$-\mathbb{R}_+$. Hence SE$-\mathbb{R}_+$ and SSE$-\mathbb{R}_+$ also preserve nonzero spectra of matrices over $\mathbb{R}$. A proof can be found in \cite[Theorem 7.4.6]{LM1995}.

\begin{definition}
  Let $A$ be an irreducible matrix with spectral radius $\lambda$. The \emph{stochasticization} of $A$ is the stochastic matrix  $\mathcal{S}(A)=\frac{1}{\lambda}D^{-1}AD$, where $D$ is the diagonal matrix whose vector of diagonal entries is the right eigenvector of $A$ corresponding to the eigenvalue $\lambda$.
\end{definition}
\section{Strong Shift Equivalence within the Same Size}

We assume throughout this section that $\mathbb{U}$ is a subfield of
$\mathbb{R}$. We give sufficient conditions for positive stochastic
matrices over $\mathbb{U}$  to be similar and SSE$-\mathbb{U}_+$ to a
positive doubly stochastic matrix of the same size. We will use the
following theorem of Kim and Roush \cite{S24} (see \cite{boyle2013path} for
a thorough exposition and generalizations).

\begin{theorem}\label{theorem:1}
Let $A$ and $B$ be positive $n\times n$ matrices over $\mathbb{U}$. Suppose $A$ and $B$ are similar over $\mathbb{U}$ and  $(A_t)_{0\leq t\leq 1}$ is a path of positive, real $n\times n$ matrices, all in the same similarity class over $\mathbb{R}$, from $A=A_0$ to $B=A_1$. Then $A$ and $B$ are SSE$-\mathbb{U}_+$ through matrices of the same size.
\end{theorem}

We begin with the following lemma.
\begin{lemma}\label{lemma:1}
  Let $P$ be an $n\times n$ positive stochastic matrix. Then $$P+J_v(I_n-P)$$ is similar to $P$ for all $v\in\Delta^{n-1}_+$.
\begin{proof}
Let $l$ be the left Perron eigenvector of $P$. It is straightforward to verify that
$$
J_lP=PJ_l=J_l,\,\, PJ_v=J_v,\,\, J_v^2=J_v,\,\, \text{ and  } J_vJ_l=J_l.
$$
Define $X=I_n-J_l-J_v$. A computation shows that $X^2=I_n$. Thus, using the above relations,
$$
XPX^{-1}=(P-J_l-J_vP)X = P+J_v(I_n-P).
$$
\end{proof}
\end{lemma}

We now in a position to prove the main result of this section.

\begin{theorem}\label{theorem:185}
  Let $P$ be an $n\times n$ positive stochastic matrix over $\mathbb{U}$. If $P+J_n(I_n-P)$ is positive then $P$ is similar and SSE$-\mathbb{U}_+$ to a positive doubly stochastic matrix through matrices of the same size.
  \begin{proof}
Let $v=\frac{1}{n}j_n$ and $Q=P+J_n(I_n-P)$. Then $Q$ is similar to $P$. A computation shows that $vQ=v$. Thus $Q$ is positive and doubly stochastic. Next, we show that $Q$ is SSE$-\mathbb{U}_+$ to $P$. Let $l$ be the left Perron eigenvector of $P$. Define $l(t)=(1-t)l+tv$ and $P_t=P+J_{l(t)}(I_n-P)$ for all $t\in [0,1]$. Using the fact that $J_lP=J_l$ and $J_{l(t)}=(1-t)J_l+tJ_n$, we get $P_t=(1-t)P+tQ$ for all $t\in [0,1]$. Hence $P_t$ is positive for all $t\in [0,1]$. By Theorem \ref{theorem:1}, $P$ and $Q$ are SSE$-\mathbb{U}_+$ through matrices of the same size.
\end{proof}
\end{theorem}

\begin{corollary}
  Every  positive $2\times 2$  stochastic matrix over $\mathbb{U}$ is similar and SSE$-\mathbb{U}_+$ to a  positive doubly stochastic matrix through matrices of the same size.
  \begin{proof}
 Let $P=\left(
          \begin{array}{cc}
            a & 1-a \\
            1-b & b \\
          \end{array}
        \right)
 $ where $a,b\in(0,1)$. Then
 $$
 P+J_2(I_2-P) =\dfrac{1}{2} \left(
                  \begin{array}{cc}
                    a+b & 2-a-b \\
                    2-a-b & a+b \\
                  \end{array}
                \right)
 $$
 is positive and hence the result follows from Theorem \ref{theorem:185}.
  \end{proof}
\end{corollary}

\begin{corollary}
  Every positive stochastic matrix of rank one is similar and SSE$-\mathbb{R}_+$ to a positive doubly stochastic matrix through matrices of the same size.
\end{corollary}
\begin{proof}
Any positive $n\times n$ stochastic matrix of rank one is of the form $J_l$ for some $l\in\Delta_+^{n-1}$. Since $J_l+J_n(I_n-J_l)=J_n$, $J_l$ is similar and SSE$-\mathbb{R}_+$ to $J_n$ by Theorem \ref{theorem:185}.
\end{proof}

\begin{remark}
  For a positive $n\times n$ stochastic matrix $P=(p_{ij})$, the following conditions are equivalent:
  \begin{enumerate}
    \item $P+J_n(I_n-P)$ is positive.
    \item $\displaystyle \sum_{k=1}^np_{kj}<1+np_{ij}$ for all $i,j\in\{1,2,\ldots,n\}$.
     \item $\displaystyle \sum_{i=1}^np_{ij}<1+n\min_{1\leq i\leq n} p_{ij}$ for all $j\in\{1,2,\ldots,n\}$.
  \end{enumerate}
  We will use the third condition to prove the following results.
\end{remark}

\begin{corollary}\label{theorem:186}
  Let $P=(p_{ij})$ be an $n\times n$ positive stochastic matrix over $\mathbb{U}$ where $n\geq 2$. Suppose that one of the following conditions holds:
\begin{enumerate}
\item $\displaystyle \max_{1\leq i\leq n}p_{ij}-\min_{1\leq i\leq n}p_{ij}<\dfrac{1}{n-1}$ for all $j\in\{1,2,\ldots,n\}$.

\item $\displaystyle \max_{1\leq i,j\leq n}p_{ij}-\min_{1\leq i,j\leq n}p_{ij}<\dfrac{1}{n}$.

\item $\displaystyle \min_{1\leq i,j\leq n} p_{ij}> \dfrac{1}{n}-\dfrac{1}{n^2}$.
\end{enumerate}
Then $P$ is similar and SSE$-\mathbb{U}_+$ to a positive doubly stochastic matrix through matrices of the same size.
\begin{proof}\hspace{1cm}
\begin{enumerate}
 \item For each $j\in\{1,2,\ldots,n\}$, we have
    \begin{align*}
      \sum_{i=1}^np_{ij} &= \min_{1\leq i\leq n}p_{ij} + \sum_{i=1}^np_{ij}-\min_{1\leq i\leq n}p_{ij} \\
             &\leq  \min_{1\leq i\leq n}p_{ij} + (n-1) \max_{1\leq i\leq n}p_{ij} \\
             &< \min_{1\leq i\leq n}p_{ij} +1+ (n-1) \min_{1\leq i\leq n}p_{ij} \\
             &=1+n\min_{1\leq i\leq n}p_{ij}.
    \end{align*}

\item   For each $j\in\{1,2,\ldots,n\}$, we have
    \begin{align*}
      \sum_{i=1}^np_{ij} &\leq  n\max_{1\leq i,k\leq n}p_{ik} \\
             &< 1+n\min_{1\leq i,k\leq n}p_{ik}\\
             &\leq 1+n\min_{1\leq i\leq n}p_{ij}.
    \end{align*}

 \item  If $p_{ij}\geq\frac{2}{n}-\frac{1}{n^2}$ for some $i,j\in\{1,2,\ldots,n\}$ then
    \begin{align*}
      \sum_{k=1}^n p_{ik} &> (n-1)\left(\dfrac{1}{n}-\dfrac{1}{n^2}\right)+ p_{ij} \\
             &\geq \dfrac{(n-1)^2}{n^2}+\dfrac{2n-1}{n^2} \\
             &= 1
    \end{align*}
which is a contradiction. Thus $\frac{1}{n}-\frac{1}{n^2}< p_{ij}<\frac{2}{n}-\frac{1}{n^2}$ for all $i,j\in\{1,2,\ldots,n\}$. For any $j\in\{1,2,\ldots,n\}$, we have
\begin{align*}
 \sum_{i=1}^n p_{ij}  &< n\left(\dfrac{2}{n}-\dfrac{1}{n^2}\right) \\
                                       &= 1+n\left(\dfrac{1}{n}-\dfrac{1}{n^2}\right) \\
                                       &<1+n\min_{1\leq i\leq n}p_{ij}.
 \end{align*}
\end{enumerate}
\end{proof}
\end{corollary}

\begin{remark}
  For similarity, the third condition in Corollary \ref{theorem:186} improves the sufficient condition $p_{ij}>\frac{1}{n+1}$ for all $i,j\in\{1,2,\ldots, n\}$ observed by Johnson in \cite[Observation 4]{johnson1981row}.
\end{remark}

\begin{lemma} \label{lemma:2}
 $P$ is strong shift equivalent to $Q$ iff $P^T$ is strong shift equivalent to $Q^T$.
 \begin{proof}
    $P$ is strong shift equivalent to $Q$ via $(R_1,S_1),(R_2,S_2),\ldots,(R_n,S_n)$ iff $P^T$ is strong shift equivalent to $Q^T$ via $(S_1^T,R_1^T),(S_2^T,R_2^T),\ldots,(S_n^T,R_n^T)$.
\end{proof}
\end{lemma}

\begin{theorem}
 Let $P$ be an $n\times n$ positive stochastic matrix over $\mathbb{U}$  with the left Perron eigenvector $l=(l_1,l_2,\ldots,l_n)$. If
 $$
 \sum_{k=1}^n\tfrac{l_i}{l_k} p_{ik}<1+n\tfrac{l_i}{l_j}p_{ij}
 $$
 for all $i,j\in\{1,2,\ldots,n\}$, then $P$ is similar and SSE$-\mathbb{U}_+$ to a positive doubly stochastic  matrix through matrices of the same size.
\end{theorem}
\begin{proof}
  Let $D=\text{diag}(l_1,l_2,\ldots,l_n)$ and $\widetilde{P} = \mathcal{S}(P^T)=D^{-1}P^TD$. Note that $\widetilde{p}_{ij}=\frac{l_jp_{ji}}{l_i}$ for all $i,j\in\{1,2,\ldots,n\}$. By assumption, $\widetilde{P}+J_n(I_n-\widetilde{P})$ is positive. By Theorem \ref{theorem:185}, $\widetilde{P}$ is similar and SSE$-\mathbb{U}_+$ to a positive doubly stochastic matrix, say $Q$, through matrices of the same size. Then $P^T$ is also similar and SSE$-\mathbb{U}_+$ to $Q$ through matrices of the same size. By Lemma \ref{lemma:2}, $P$ is similar and SSE$-\mathbb{U}_+$ to $Q^T$ through matrices of the same size. Since $Q$ is doubly stochastic, $Q^T$ is also doubly stochastic and the proof is completed.
\end{proof}

It is not true in general that every positive stochastic matrix is SSE$-\mathbb{R}_+$ to a positive doubly stochastic matrix of the same size, because there are positive stochastic matrices whose nonzero spectra cannot be the nonzero spectra of doubly stochastic matrices of the same size. An example can be found in \cite{johnson1981row}, which we will reprove here. Firstly, we reprove the following result of Johnson \cite{johnson1981row}.
\begin{proposition}\label{theorem:184}
There is no $3\times 3$ doubly stochastic matrix with the characteristic polynomial $t(t-1)(t+1)$.
\begin{proof}
Suppose there is such a matrix $$A=\left(
                                     \begin{array}{ccc}
                                       a & b & 1-a-b \\
                                       c & d & 1-c-d \\
                                       1-a-c & 1-b-d & a+b+c+d-1 \\
                                     \end{array}
                                   \right).
$$
Observe that $\det(A)=0$ and $\text{Tr}(A)=0$.
Since $A$ is nonnegative and $\text{Tr}(A)=0,$ we have $a=d=a+b+c+d-1=0$. Thus $b+c=1$. Then $A$ can be rewritten as $$A=\left(
                                     \begin{array}{ccc}
                                       0 & b & c \\
                                       c & 0 & b \\
                                       b & c & 0 \\
                                     \end{array}
                                   \right).
$$ Hence $b^3+c^3=\det{(A)}=0$. This implies $b=c=0$ which is a contradiction.
\end{proof}
\end{proposition}

Next, we define for any $n\in\mathbb{N}$ the matrix
$$A_n=\dfrac{1}{n+2}\left(
        \begin{array}{ccc}
          1 & n & 1 \\
          n & 1 & 1 \\
          n & 1 & 1 \\
        \end{array}
      \right)
.$$
Suppose that there is a sequence of $3\times 3$ doubly stochastic matrices $\{B_n\}$ such that $B_n$ and $A_n$ are similar for all $n\in\mathbb{N}$. By compactness, $\{B_n\}$ has a convergent subsequence $\{B_{n_k}\}$. Suppose that $\{B_{n_k}\}$ converges to a matrix $B$. Then $B$ is doubly stochastic since the set of doubly stochastic matrices is closed. For any $n\in\mathbb{N}$, the characteristic polynomial of $A_n$ is $p_n(t)=t(t-1)(t+\frac{n-1}{n+2})$ which converges to $t(t-1)(t+1)$ as $n\to\infty$. Thus $B$ must have the characteristic polynomial $t(t-1)(t+1)$ which is a contradiction. So there must be some matrix $A_{n_0}$ which is not similar to a doubly stochastic matrix. Since strong shift equivalence preserves the Jordan form away from zero (and in this case it is the Jordan form), $A_{n_0}$ is not SSE$-\mathbb{R}_+$ to a $3\times 3$ doubly stochastic matrix.
\section{General Strong Shift Equivalence}

Throughout this section, we assume that $\mathbb{U}$ is a subring of $\mathbb{R}$ containing $\mathbb{Q}$. In contrast to similarity, we prove that every positive stochastic matrix over $\mathbb{U}$ is SSE$-\mathbb{U}_+$ to a positive doubly stochastic matrix. We begin with a technical lemma.

\begin{lemma} \label{theorem:181}

Every positive stochastic matrix over $\mathbb{U}$ is similar and SSE$-\mathbb{U}_+$ to a positive stochastic matrix  whose left Perron eigenvector is rational.
\begin{proof} Let $P$ be an $n\times n$ positive stochastic matrix with the left Perron eigenvector $l=(l_1,l_2,\ldots,l_n)$. For each $k\in\mathbb{N}$, let $r_k\in\mathbb{Q}_+^{n-1}$ be such that $r_{kj}\leq l_j$ for all $j\in\{1,2,\ldots,n-1\}$ and $\displaystyle\lim_{k\to\infty}r_k=(l_1,l_2,\ldots,l_{n-1})$. Define
$$M_k=\left(
              \begin{array}{ccccc}
                \frac{l_1}{r_{k1}} & 0 & \cdots & 0 & 1-\frac{l_1}{r_{k1}} \\
                0 & \frac{l_2}{r_{k2}} & \cdots & 0 & 1 - \frac{l_2}{r_{k2}} \\
                \vdots & \vdots & \ddots & \vdots & \vdots \\
                0 & 0 & \cdots & \frac{l_{n-1}}{r_{k,n-1}} & 1 - \frac{l_{n-1}}{r_{k,n-1}} \\
                0 & 0 & \cdots & 0 & 1 \\
              \end{array}
            \right)
$$ and $P_k=M_kPM_k^{-1}$. Note that $M_k\to I_n$ and $P_k\to P$ as $k\to\infty$ and
$$M_k^{-1}=\left(
              \begin{array}{ccccc}
                \frac{r_{k1}}{l_1} & 0 & \cdots & 0 & 1-\frac{r_{k1}}{l_1} \\
                0 & \frac{r_{k2}}{l_2} & \cdots & 0 & 1 - \frac{r_{k2}}{l_2} \\
                \vdots & \vdots & \ddots & \vdots & \vdots \\
                0 & 0 & \cdots & \frac{r_{k,n-1}}{l_{n-1}} & 1 - \frac{r_{k,n-1}}{l_{n-1}} \\
                0 & 0 & \cdots & 0 & 1 \\
              \end{array}
            \right)$$
is nonnegative for all $k\in\mathbb{N}$. Choose $N$ such that $M_NP$ is positive. Let $\widehat{l}=(\widehat{l}_1,\widehat{l}_2,\ldots,\widehat{l}_n)$ where $\widehat{l}_j=r_{Nj}$ for $j\in\{1,2,\ldots,n-1\}$ and $\displaystyle\widehat{l}_n=1-\sum_{j=1}^{n-1} r_{Nj}$. Then $\widehat{l}$ is rational, $\widehat{l}M_N=l$, and
\begin{align*}
\widehat{l}P_N &=\widehat{l}M_NPM_N^{-1} \\
&=lPM_N^{-1} \\
&= lM_N^{-1} \\
&=\widehat{l}.
\end{align*}
Thus $\widehat{l}$ is the left Perron eigenvector of $P_N$. Since $M_NP$ is positive, $P$ is similar and SSE$-\mathbb{U}_+$ to $P_N$.
\end{proof}
\end{lemma}

\begin{theorem} \label{theorem:182}
Every  positive stochastic matrix over $\mathbb{U}$ is SSE$-\mathbb{U}_+$ to a positive doubly stochastic matrix over $\mathbb{U}$.
\begin{proof}
Let $P$ be an $n\times n$ stochastic matrix over $\mathbb{U}$. By Lemma~\ref{theorem:181}, we can assume that the left Perron eigenvector of $P$ is rational, namely
$$
l=\left(\frac{r_1}{s_1},\frac{r_2}{s_2},\ldots,\frac{r_n}{s_n}\right)
$$
where $r_i,s_i\in\mathbb{N}$ for all $i\in\{1,2,\ldots,n\}$. Let
$$
M=\text{lcm}(s_1,s_2,\ldots,s_n).
$$
Then $l$ can be written as
$$
l=\frac{1}{M}\left(m_1,m_2,\ldots,m_n\right)
$$
where $m_i=\frac{r_iM}{s_i}\in\mathbb{N}$ for all $i\in\{1,2,\ldots,n\}$. If $m_1\neq 1$, we perform an elementary column splitting on the first column of $P$ as follows:
$$P^{(1)}=\left(
           \begin{array}{ccccc}
             \frac{1}{m_1}p_{11} & (1-\frac{1}{m_1})p_{11} & p_{12} & \cdots & p_{1n} \\
             \frac{1}{m_1}p_{11} & (1-\frac{1}{m_1})p_{11} & p_{12} & \cdots & p_{1n} \\
             \frac{1}{m_1}p_{21} & (1-\frac{1}{m_1})p_{21} & p_{22} & \cdots & p_{2n} \\
             \vdots & \vdots & \vdots & \ddots & \vdots \\
             \frac{1}{m_1}p_{n1} & (1-\frac{1}{m_1})p_{n1} & p_{n2} & \cdots & p_{nn} \\
           \end{array}
         \right).
$$
The left Perron eigenvector of $P^{(1)}$ is
$$
l^{(1)}=\frac{1}{M}\left(1,m_1-1,m_2,\ldots,m_n\right).
$$
If $m_1-1\neq 1$ we perform an elementary column splitting on the second column of $P^{(1)}$ by splitting the second column of $P^{(1)}$ as $\frac{1}{m_1-1}C_2^{(1)}$ and $(1-\frac{1}{m_1-1})C_2^{(1)}$ where $C_2^{(1)}$ is the second column of $P^{(1)}$. Suppose $P^{(2)}$ is the matrix after splitting $P^{(1)}$. Then the left Perron eigenvector of $P^{(2)}$ is
$$
l^{(2)}=\frac{1}{M}\left(1,1,m_1-2,\ldots,m_n\right).
$$
Continuing in this manner, we finally get an $M\times M$ matrix $P^{(k)}$ whose the left Perron eigenvector $l^{(k)}$ is $\frac{1}{M}j_M$ for some $k\in\mathbb{N}$. Note that $P^{(i)}$ is stochastic for all $i\in\{1,2,\ldots,k\}$. Therefore, $P^{(k)}$ is doubly stochastic. This completes the proof.
\end{proof}
\end{theorem}

\begin{example}
  Let $P=\dfrac{1}{10}\left(
                       \begin{array}{ccc}
                         7 & 2 & 1 \\
                         2 & 7 & 1 \\
                         2 & 2 & 6 \\
                       \end{array}
                     \right)
  $. Then $l=\frac{1}{5}(2,2,1)$ is the left Perron eigenvector of $P$. Using the process in the proof of Theorem \ref{theorem:182}, we obtain
$$
P^{(1)}=\frac{1}{20}\left(
                              \begin{array}{cccc}
                                7 & 7 & 4 & 2 \\
                                7 & 7 & 4 & 2 \\
                                2 & 2 & 14 & 2 \\
                                2 & 2 & 4 & 12 \\
                              \end{array}
                            \right)
                      \text{ and }
P^{(2)}=\frac{1}{20}\left(
                      \begin{array}{ccccc}
                        7 & 7 & 2 & 2 & 2 \\
                        7 & 7 & 2 & 2 & 2 \\
                        2 & 2 & 7 & 7 & 2 \\
                        2 & 2 & 7 & 7 & 2 \\
                        2 & 2 & 2 & 2 & 12 \\
                      \end{array}
                    \right).
$$
\end{example}

\begin{corollary}\label{theorem:183}
If $\mathbb{U}$ is a subfield of $\mathbb{R}$, then the set of nonzero spectra of positive doubly stochastic matrices over $\mathbb{U}$ and the set of nonzero spectra of primitive stochastic matrices over $\mathbb{U}$ with positive trace coincide.
\begin{proof}
Let $A$ be a primitive stochastic matrix over $\mathbb{U}$ with positive trace. By \cite[Proposition B.3]{boyle2013path}, $A$ is SSE$-\mathbb{U}_+$ to a positive matrix $\widetilde{A}$ over $\mathbb{U}$. The right eigenvector of $\widetilde{A}$ corresponding to the eigenvalue $1$ is a positive vector over $\mathbb{U}$. Thus $\mathcal{S}(\widetilde{A})$ is a positive stochastic matrix over $\mathbb{U}$. By Theorem \ref{theorem:182}, $\mathcal{S}(\widetilde{A})$ is SSE$-\mathbb{U}_+$ to a positive doubly stochastic matrix $B$ over $\mathbb{U}$. This implies $A$ is SSE$-\mathbb{U}_+$ to $B$. Thus $A$ and $B$ have identical nonzero spectrum.
\end{proof}
\end{corollary}

\begin{remark} \label{rationalsizeremark}
The bound on the size of a doubly stochastic matrix obtained from Theorem~\ref{theorem:182} is still unknown. We mention also that a primitive $n\times n$ matrix over $\mathbb U$ is SSE-$\mathbb U_+$ to a positive matrix of size at most $2n^2 \times 2n^2$ \cite[Proposition B.3]{boyle2013path}.
\end{remark}

\section{Unboundedness of Lags of SSE$-\mathbb{R}_+$}

The purpose of this section is to provide the following example.

\begin{theorem}\label{theorem:191}
For $t\in [0,1]$, define $$P_t=\frac{1}{4}\left(
                             \begin{array}{cc}
                               3+t & 1-t \\
                               1+t & 3-t \\
                             \end{array}
                           \right).
$$ For $0\leq t<1$, the matrices $P_t$ are positive, similar, and
SSE$-\mathbb{R}_+$. However, for any positive integer $L\geq 2$, there exists $t\in (0,1)$
such that there is no SSE$-\mathbb{R}_+$ of lag less than $L$ between $P_0$ and $P_t$ using
only matrices with size less than $L$.
\end{theorem}

\begin{definition}
  A (not necessarily square) matrix $A$ is called \emph{generalized row stochastic} if it is nonnegative and every row sum of $A$ is $1$.
\end{definition}

We need the following theorem for the proof of Theorem~\ref{theorem:191}.

\begin{theorem} \label{theorem:193}
Let $A$ and $B$ be respectively $m\times m$ and $n\times n$ irreducible matrices over $\mathbb{R}$. If $A$ and $B$ are $\text{ESSE}-\mathbb{R}_+$ then $\mathcal{S}(A)$ and $\mathcal{S}(B)$ are also $\text{ESSE}-\mathbb{R}_+$. Moreover, there exist generalized row stochastic matrices $R,S$ such that $\mathcal{S}(A)=RS$ and $\mathcal{S}(B)=SR$.
\begin{proof}  Since $A$ and $B$ are ESSE$-\mathbb{R}_+$, they have the same Perron eigenvalue $\lambda$. Let $v=(v_1,v_2,\ldots,v_m)\in\mathbb{R}^m$ and $w=(w_1,w_2,\ldots,w_n)\in\mathbb{R}^n$ be such that $Av=\lambda v$ and $Bw=\lambda w$. Let $D=$ diag($v_1,\ldots,v_m$) and $E=$ diag($w_1,..,w_n$). Then $\mathcal{S}(A)=\frac{1}{\lambda}D^{-1}AD$ and $\mathcal{S}(B)=\frac{1}{\lambda}E^{-1}BE$. Suppose that $A=XY$ and $B=YX$. Then
\begin{center}
$\mathcal{S}(A)=\Big(\frac{1}{\lambda}D^{-1}XE\Big)\Big(E^{-1}YD\Big)$ and $\mathcal{S}(B)=\Big(E^{-1}YD\Big)\Big(\frac{1}{\lambda}D^{-1}XE\Big)$.
\end{center}
Thus $\mathcal{S}(A)$ and $\mathcal{S}(B)$ are ESSE$-\mathbb{R}_+$. Next, suppose that $\mathcal{S}(A)=UV$ and $\mathcal{S}(B)=VU$. Since $\mathcal{S}(A)U=U\mathcal{S}(B)$, we have $\mathcal{S}(A)Uj_n^T=U\mathcal{S}(B)j_n^T=Uj_n^T.$ Thus $Uj_n^T$ is a right eigenvector of $\mathcal{S}(A)$ corresponding to an eigenvalue $1$ and hence $Uj_n^T=\alpha j_m^T$ for some $\alpha >0$. Similarly, $Vj_m^T=V\mathcal{S}(A)j_m^T=\mathcal{S}(B)Vj_m^T$, so $Vj_m^T=\beta j_n^T$ for some $\beta > 0$. Let $R=\frac{1}{\alpha}U$ and $S=\frac{1}{\beta}V$. Then $Rj_n^T=\frac{1}{\alpha}Uj_n^T=j_m^T$ and $Sj_m^T=\frac{1}{\beta}Vj_m^T=j_n^T$. Thus $R,S$ are generalized row stochastic matrices. Furthermore, we have $\mathcal{S}(A)=UV=(\alpha\beta) RS$ and $\mathcal{S}(B)=VU=(\alpha\beta) SR$. Note that
\begin{align*}
m &=j_m\mathcal{S}(A)j_m^T \\
&=\alpha\beta j_mRSj_m^T \\
&=\alpha\beta j_mRj_n^T \\
&=\alpha\beta j_mj_m^T \\
&=m\alpha\beta.
\end{align*}
Thus $\alpha\beta = 1$ and hence $\mathcal{S}(A)=RS$ and $\mathcal{S}(B)=SR$.
\end{proof}
\end{theorem}

\begin{proof}[Proof of Theorem~\ref{theorem:191}]
The similarity holds because $\text{Tr}(P_t)=6, \det(P_t)=8$ for all $0\leq t<1$. By Theorem~\ref{theorem:1}, $P_t$ and $P_0$ are SSE$-\mathbb{R}_+$
for all $0\leq t<1$. It is well-known that strong shift equivalence preserves irreducibility \cite[Proposition 7.4.1]{LM1995}. So $P_0$ and $P_1$ are not
SSE$-\mathbb{R}_+$ because $P_0$ is irreducible whereas $P_1$ is reducible. Next, suppose that $P_0$ and $P_t$ are SSE over $\mathbb{R}_+$ via
$2\times 2$ matrices with lag $l\leq k$ and size $n\leq k$ for all $t\in (0,1)$. Without loss of generality, we assume that the lag $l=k$ for all $t\in (0,1)$. For
each $t\in (0,1)$ we have a chain of ESSEs over $\mathbb{R}_+$, $P_0,A_1(t), \ldots, A_{k-1}(t),P_t$, together with a chain of intermediate matrices
$(R_1(t),S_1(t)),\ldots,(R_k(t),S_k(t))$. Since $P_t$ is positive for each $0\leq t<1$, each $A_i(t)$ has a unique maximal irreducible submatrix, say
$\hat{A}_i(t)$. The given SSE restricts to an SSE of the $\hat{A}_i(t)$. So, without loss of generality, we assume $\hat{A}_i(t)=A_i(t)$. By using
Theorem~\ref{theorem:193}, we can assume that $A_i(t),R_i(t),S_i(t)$ are generalized row stochastic for all $i\in\{1,2, \ldots,k\}$ and all $t\in (0,1)$.
Then all matrices are bounded (by 1), so there is a subsequence $t_n\to 1$ such that, for each $i\in\{1,2, \ldots,k\}$, $A_i(t_n)\to A_i, R_i(t_n)\to
R_i,S_i(t_n)\to S_i$ for some nonnegative matrices $A_i,R_i,S_i$. But then we get a strong shift equivalence over $\mathbb{R}_+$ between $P_0$ and $P_1$ which is a contradiction.
\end{proof}

\begin{question}
  Does Theorem \ref{theorem:191} remain true if the constraint \lq\lq using only matrices with size less than $L$\rq\rq is deleted?
\end{question}

\begin{definition}
A \textit{semialgebraic subset} of $\mathbb{R}^n$ is a subset of points in $\mathbb{R}^n$ which is the solution set of a
boolean combination of polynomial equations and inequalities with real coefficients.
\end{definition}

In contrast to Theorem \ref{theorem:191}, we have the following results.

\begin{theorem}\label{sseclasses}
Let $A$ be an $n\times n$ primitive matrix with positive trace. The collection of positive $n\times n$ matrices similar over $\mathbb{R}$ to $A$ contains only finitely many SSE$-\mathbb{R}_+$ classes.
\begin{proof}
Let $\mathcal{M}_+(A)=\{X \mid X \text{ is } n\times n,\text{ positive and similar to } A\}$. We will prove that $\mathcal{M}_+(A)$ contains only finitely many connected components. It suffices to prove that $\mathcal{M}_+(A)$ is a semialgebraic set since it is well-known that a semialgebraic set has finitely many connected components \cite[Theorem 2.4.4]{bochnak1998real}. Let $p_A(t)$ be the characteristic polynomial of $A$. Suppose that $p_A(t)=\prod_{k=1}^m (q_k(t))^{j_k}$ where the $q_k$ are irreducible and distinct, and $j_k\in\mathbb{N}$ for all $k\in\{1,2,\ldots,m\}$. Then $X\in\mathcal{M}_+(A)$ if and only if
\begin{enumerate}
  \item $x_{ij}>0$ for all $i,j\in\{1,2,\ldots,n\}$, and
  \item $\text{rank}(q_k(X))^{j}=\text{rank}(q_k(A))^{j}$ for all $k\in\{1,2,\ldots,m\}$ and $j\in\{1,2,\ldots,j_k\}$.
\end{enumerate}
That a matrix $M$ has a given rank $r$ is equivalent to $r\times r$ being the size of the largest submatrix of $M$ with nonzero determinant. This is a semialgebraic condition on $M$. Thus $\mathcal{M}_+(A)$ is semialgebraic. By Theorem \ref{theorem:1}, $\mathcal{M}_+(A)$ contains only finitely many SSE$-\mathbb{R}_+$ classes.
\end{proof}
\end{theorem}

\begin{remark}
  Theorem \ref{sseclasses} (taken from \cite{chuysurichay2011positive}) was generalized to any positive matrix over a dense subring of $\mathbb{R}$ in \cite[Theorem 5.16]{boyle2013path}.
\end{remark}

\begin{theorem}\label{seclasses}
Let $A$ be an $n\times n$ primitive matrix with positive trace. The collection of positive $n\times n$ matrices SE$-\mathbb{R}_+$ to $A$ contains only finitely many SSE$-\mathbb{R}_+$ classes.
\begin{proof}
Let $\text{SE}(n,A)=\{X \mid X \text{ is } n\times n,\text{ positive and SE}-\mathbb{R}_+ \text{ to } A \}$. Since SE$-\mathbb{R}_+$ preserves the Jordan form away from zero \cite[Theorem 7.4.6]{LM1995}, any $n\times n$ positive matrix which is SE$-\mathbb{R}_+$ to $A$ must have $J^{\times}(A)$ as its Jordan form away from zero. There are only finitely many Jordan forms with the same $J^{\times}(A)$. Each Jordan type contributes only a finite number of SSE$-\mathbb{R}_+$ classes by Theorem \ref{sseclasses}. Therefore, $\text{SE}(n,A)$ contains finitely many SSE$-\mathbb{R}_+$.
\end{proof}
\end{theorem}

\section{Acknowledgements}
This paper includes results from the doctoral dissertation of the author \cite{chuysurichay2011positive}, completed at the University of
 Maryland, College Park under the supervision of Mike Boyle.
\bibliographystyle{amsplain}
\bibliography{SSEarxiv}
\end{document}